\documentclass[10pt]{amsart}

\usepackage{amsmath}
\usepackage{amssymb}
\usepackage{verbatim}
\usepackage{enumerate}

\usepackage{setspace}
\usepackage{fancyhdr}
\usepackage{indentfirst}
\usepackage{amsthm}
\usepackage{mathrsfs}

\theoremstyle{definition}
\newtheorem{thm}{Theorem}[section]
\newtheorem{lem}[thm]{Lemma}

\newtheorem{dfn}[thm]{Definition}
\newtheorem{cor}[thm]{Corollary}

\newtheorem{ntn}[thm]{Notation}
\newtheorem{exa}[thm]{Example}

\numberwithin{equation}{section}

\newcounter{TmpEnumi}

\newcommand{\af}{\alpha}

\newcommand{\gm}{\gamma}
\newcommand{\dt}{\delta}
\newcommand{\ep}{\varepsilon}

\newcommand{\te}{\theta}
\newcommand{\ld}{\lambda}
\newcommand{\sm}{\sigma}

\newcommand{\ph}{\varphi}
\newcommand{\rh}{\rho}

\newcommand{\ta}{\tau}

\newcommand{\Z}{{\mathbb{Z}}}
\newcommand{\R}{{\mathbb{R}}}
\newcommand{\C}{{\mathbb{C}}}
\newcommand{\N}{{\mathbb{Z}}_{> 0}}

\newcommand{\dist}{{\operatorname{dist}}}

\newcommand{\Aut}{{\operatorname{Aut}}}

\newcommand{\andeqn}{\,\,\,\,\,\, {\mbox{and}} \,\,\,\,\,\,}

\newcommand{\norm}[1]{\left\| #1 \right\|}
\newcommand{\bnorm}[1]{\big\| #1 \big\|}

\newcommand{\set}[1]{\left\{ #1 \right\}}
\newcommand{\bset}[1]{\big\{ #1 \big\}}

\newcommand{\cuntz}{\precsim}
\newcommand{\JS}{{\mathcal{Z}}}

\newcommand{\ca}{C*-algebra}
\newcommand{\uca}{unital \ca}

\newcommand{\alignInd}{\hspace*{3em} }

\newcommand{\hme}{homeomorphism}
\newcommand{\mh}{minimal homeomorphism}
\newcommand{\cms}{compact metric space}

\title{Centrally Large Subalgebras and Tracial
 ${\mathcal{Z}}$-Absorption}

\author{Dawn Archey, Julian Buck, and N.~Christopher Phillips}

\date{5~August 2016}

\address{Department of Mathematics and Software Engineering,
University of Detroit Mercy,
4001 West McNichols Road,
Detroit MI 48221-3038,
USA}

\address{Department of Mathematics, Francis Marion University,
 PO Box 100547, Florence SC 29502, USA.}

\address{Department of Mathematics, University  of Oregon,
       Eugene OR 97403-1222, USA.}

\subjclass[2010]{Primary 46L05, 46L55;
 Secondary 46L35.}
\thanks{This material is partially based upon work
  of the third author supported by the
  US National Science Foundation under
  Grants DMS-0701076, DMS-1101742, and DMS-1501144.
  Some of the work was done during a visit by the third author to
  the Institut Mittag-Leffler,
  and he is grateful to the Institut Mittag-Leffler
  for its hospitality and support.}

\begin{document}

\begin{abstract}
Let $A$ be a simple infinite dimensional
stably finite unital C*-algebra,
and let $B$ be a centrally large subalgebra of~$A$.
We prove that if $A$ is tracially ${\mathcal{Z}}$-absorbing
if and only if $B$ is tracially ${\mathcal{Z}}$-absorbing.
If $A$ and $B$ are also separable and nuclear,
we prove that $A$ is ${\mathcal{Z}}$-absorbing
if and only if $B$ is ${\mathcal{Z}}$-absorbing.
\end{abstract}

\maketitle

Let ${\mathcal{Z}}$ be the Jiang-Su algebra.
In this paper we prove that if $A$ is a simple infinite dimensional
stably finite unital C*-algebra and $B \subset A$ is a
centrally large subalgebra in the sense of~\cite{ArPh},
then $A$ is tracially ${\mathcal{Z}}$-absorbing
in the sense of~\cite{HirOr2013}
if and only if $B$ is tracially ${\mathcal{Z}}$-absorbing.
If, in addition, $A$ and $B$ are separable and nuclear,
then ${\mathcal{Z}} \otimes A \cong A$
if and only if ${\mathcal{Z}} \otimes B \cong B$.
(The actual hypotheses are slightly weaker;
see Theorem~\ref{tZabsGoesUp},
Theorem~\ref{T_5715_ZabsGoesUp},
Theorem~\ref{T_6208_ZDown},
and Corollary~\ref{C_6208_ZabsGoesDown}.)

Applications will appear elsewhere.
The main ones so far are as follows.
First,
let $X$ be an infinite \cms,
and let $h \colon X \to X$ be a \mh{}
with mean dimension zero.
Then $C^* (\Z, X, h)$ is $\JS$-stable.
This is proved by Elliott and Niu in~\cite{EN2},
and Theorem~\ref{T_5715_ZabsGoesUp} plays a key role.
Second,
Theorem~\ref{T_5715_ZabsGoesUp}
has been used in~\cite{ArBkPh} to prove
$\JS$-stability of crossed products
$C^* \big( \Z, \, C (X, D), \, \af \big)$
when $D$ is simple, unital, and nuclear,
the automorphism $\af \in \Aut (C (X, D))$
``lies over'' a \mh{} of~$X$
(in interesting cases, with large mean dimension),
and $\JS$-stability of the crossed product
comes from~$D$ rather than from the action of $\Z$ on~$X$.
Third,
David Kerr has proved (\cite{Ke})
that if $G$ is a countable infinite amenable group,
then there is a free minimal action of $G$ on
the Cantor set $X$ with a system of Rokhlin towers which is
good enough to construct an AF subalgebra of $C^{*} (G, X)$
that is centrally large in the sense of
Definition~\ref{LargeSubalg}.
It follows from Theorem~\ref{T_5715_ZabsGoesUp}
that $C^{*} (G, X)$ is $\JS$-stable.
We also mention Wei Sun's
``generalized higher dimensional noncommutative tori''~\cite{Su}.
Some of these are isomorphic to centrally large subalgebras
of crossed products by rotation actions
of $\Z$ on $(S^1)^n$.
Corollary~\ref{C_6208_ZabsGoesDown} implies that these
algebras are $\JS$-stable,
and from this fact one can deduce stable rank one
and sometimes deduce real rank zero.
See Example~\ref{E_6804_WeiSunEx}.

Large subalgebras were introduced in~\cite{PhLarge}.
They are an abstraction of an idea whose initial form
appeared in~\cite{Pt1},
and one of the main examples,
the orbit breaking subalgebra of a crossed product
by a \mh{}
(see Theorem~7.10 of~\cite{PhLarge})
is a generalization of the construction of~\cite{Pt1}.
It was shown in~\cite{PhLarge}
that if $B$ is stably large in~$A$,
then $A$ and $B$ have many properties in common.
For example, they have the same tracial states
(Theorem~6.2 of~\cite{PhLarge}),
the same quasitraces
(Proposition~6.9 of~\cite{PhLarge}),
the same purely positive part of the Cuntz semigroup
(Theorem~6.8 of~\cite{PhLarge}),
and the same radius of comparison
(Theorem~6.14 of~\cite{PhLarge}).
Centrally large subalgebras,
introduced in~\cite{ArPh}, satisfy an extra condition;
the appropriate orbit breaking subalgebras of crossed products
by \mh{s}
are centrally large
(Theorem~7.10 of~\cite{PhLarge}
and Theorem~4.6 of~\cite{ArPh}).
If $B$ is centrally large in~$A$
and $B$ has stable rank one,
then $A$ has stable rank one
(Theorem~6.3 of~\cite{ArPh}),
and if in addition $B$ has real rank zero,
then the same is true of~$A$
(Theorem~6.3 of~\cite{ArPh}).
This paper extends the results above
by proving that tracial $\JS$-stability
and, in the separable nuclear case, $\JS$-stability,
pass between a stably centrally large subalgebra
and the containing algebra.
We do not know whether it suffices to
consider just a large subalgebra.

In Section~\ref{Sec_AppCm}
we recall the definitions and prove several
technical results related to approximate commutation relations
and order zero maps.
Corollary~\ref{C_6208_EqForCPCZ}
gives a very explicit characterization
of which linear maps from $M_n$ are
completely positive contractive with order zero.
In Section~\ref{Sec_TZAbs}
we recall the definition of tracial $\JS$-absorption
and prove the main results.

We use Section~1 of~\cite{PhLarge}
as a general reference for facts about Cuntz comparison,
although the results we use are not new there.
Also,
we will repeatedly use without comment the fact
(Proposition~5.2 of~\cite{PhLarge})
that if $B$ is a large subalgebra of a simple unital \ca~$A$,
then $B$ is simple.

\section{Centrally Large Subalgebras and Approximate
 Commutation}\label{Sec_AppCm}

We recall the definitions of large and centrally large subalgebras
(Definition~4.1 of~\cite{PhLarge}
and Definition~3.2 of~\cite{ArPh}).

\begin{dfn}\label{LargeSubalg}
Let $A$ be a simple unital infinite dimensional \ca.
A unital subalgebra $B \subset A$ is said to
be {\emph{large}} in $A$ if for every $\ep > 0$,
every $m \in \N$,
all $a_{1}, a_{2}, \ldots, a_{m} \in A$,
every $y \in B_{+} \setminus \set{0}$, and
every $x \in A_{+} \setminus \set{0}$ with $\norm{x} = 1$,
there exist
$c_{1}, c_{2}, \ldots, c_{m} \in A$ and $g \in B$ such that:
\begin{enumerate}
\item\label{LargeSubalg_1}
$0 \leq g \leq 1$.
\item\label{LargeSubalg_2}
$\norm{c_{j} - a_{j}} < \ep$ for $j = 1, 2, \ldots, m$.
\item\label{LargeSubalg_3}
$(1 - g) c_{j} \in B$ for $j = 1, 2, \ldots, m$.
\item\label{LargeSubalg_4}
$g \cuntz_{B} y$ and $g \cuntz_{A} x$.
\item\label{LargeSubalg_5}
\setcounter{TmpEnumi}{\value{enumi}}
$\norm{ (1 - g) x (1 - g) } > 1 - \ep$.
\end{enumerate}
We say that $B$ is {\emph{centrally large}} in~$A$
if we can in addition arrange that:
\begin{enumerate}
\setcounter{enumi}{\value{TmpEnumi}}
\item\label{LargeSubalg_6}
$\norm{ g a_{j} - a_{j} g } < \ep$ for $j = 1, 2, \ldots, m$.
\end{enumerate}
\end{dfn}

Lemma~\ref{L_9430_LgTower} below is a strengthening of Definition
\ref{LargeSubalg} in the following ways.
The element $g$ can be replaced by a tower of elements,
as shown in Lemma~6.1 of~\cite{ArPh}.
Additionally, by proceeding as in the
proof of Lemma~4.7 of~\cite{PhLarge} at the appropriate step,
it can be arranged that the elements $c_{j}$ which are chosen
satisfy $\norm{c_{j}} \leq \norm {a_{j}}$.
We omit the details of the proof;
see \cite{ArPh} and~\cite{PhLarge}.

\begin{lem}\label{L_9430_LgTower}
Let $A$ be an infinite dimensional simple separable unital
\ca,
and let $B \subset A$ be a centrally large subalgebra.
Then for all $m, N \in \N$,
all $a_1, a_2 \ldots, a_m \in A$,
every $\ep > 0$,
every $x \in A_{+}$ with $\| x \| = 1$,
and every $y \in B_{+} \setminus \{ 0 \}$,
there are $c_1, c_2 \ldots, c_m \in A$
for $j = 1, 2, \ldots, m$ and
$g_0, g_1, \ldots, g_N \in B$ such that:
\begin{enumerate}
\item\label{9430_LgTower-Cut1}
$0 \leq g_n \leq 1$
for $n = 0, 1, \ldots, N$
and $g_{n - 1} g_n = g_n$ for $n = 1, 2, \ldots, N$.
\item\label{9430_LgTower-Cut2}
For $j = 1, 2, \ldots, m$ we
have $\| c_j - a_j \| < \ep$.
\item\label{9430_LgTower-Cut3}
For $j = 1, 2, \ldots, m$ and
$n = 0, 1, \ldots, N$, we have $(1 - g_n) c_j \in B$.
\item\label{9430_LgTower-Cut4}
For $n = 0, 1, \ldots, N$, we
have $g_n \cuntz_B y$ and $g_n \cuntz_A x$.
\item\label{9430_LgTower-Cut5a}
For $n = 0, 1, \ldots, N$, we
have $\| (1 - g_n) x (1 - g_n) \| > 1 - \ep$.
\item\label{9430_LgTower-Cut5}
For $j = 1, 2, \ldots, m$ and
$n = 0, 1, \ldots, N$, we have $\| g_n a_j - a_j g_n \| < \ep$.
\item\label{9430_LgTower-CutNorm}
For $j = 1, 2, \ldots, m$
we have $\norm{c_{j}} \leq \norm{a_{j}}$.
\end{enumerate}
\end{lem}

As was shown in the proof
of Lemma 6.1 of~\cite{ArPh},
it is enough to take $n = 0$
in~(\ref{9430_LgTower-Cut4}) and~(\ref{9430_LgTower-Cut5a})
and $n = N$ in (\ref{9430_LgTower-Cut3}).

In the definition of tracial $\JS$-absorption,
one needs to control the norms of certain commutators
in~$A$
in terms of norms of commutators in~$B$.
The following lemma contains the basic estimate.
We will combine it with the choice $N = 1$
in Lemma~\ref{L_9430_LgTower}.

\begin{lem}\label{L-phZcjCom}
For every $\ep > 0$ there exists $\dt > 0$ such that the following
holds.
Let $A$ be a \ca,
and let $z, z_0, c, g_0, g_1 \in A$ satisfy:
\begin{enumerate}
\item\label{L_5715_CommEst_Order}
$0 \leq g_1 \leq g_0 \leq 1$.
\item\label{L_5715_CommEst_RelUnit}
$g_0 g_1 = g_1$.
\item\label{L_5715_CommEst_Norm}
$\norm{z} \leq 1$, $\norm{z_0} \leq 1$, and $\norm{c} \leq 1$.
\item\label{L_5715_CommEst_Commgc02}
$\norm{ [c, \, g_0] } < \dt$
and $\norm{ [c, \, g_1] } < \dt$.
\item\label{L_5715_CommEst_Commg2c}
$\norm{ [z_0, \, (1 - g_1) c ] } < \dt$.
\item\label{L_5715_CommEst_CommG0}
$\norm{ [z_0, \, g_0 ] } < \dt$.
\item\label{L_5715_CommEst_Dist}
$\bnorm{ z - (1 - g_0)^{1/2} z_0 (1 - g_0)^{1/2} } < \dt$.
\end{enumerate}
Then
$\norm{[z, c]} < \ep$.
\end{lem}

\begin{proof}
Let $\ep > 0$ be given.
Apply Lemma 2.5 of~\cite{ArPh} with $\frac{\ep}{11}$
in place of $\ep$
and with the function $f (\ld) = \ld^{1/2}$,
obtaining $\dt_0 > 0$ such that
whenever $D$ is a \ca{} and $a, b \in D$
satisfy the relations $\| [a, b ] \| < \dt_0$,
$0 \leq a \leq 1$, and $\norm{b} \leq 1$,
then $\| [a^{1/2}, b ] \| < \frac{\ep}{11}$.
Set $\dt = \min \big( \frac{\ep}{11}, \dt_0 \big)$.
This is the number whose existence is asserted in the lemma.

Now let $A$ be a \ca,
and let $z, z_0, c, g_0, g_1 \in A$ satisfy
(\ref{L_5715_CommEst_Order})--(\ref{L_5715_CommEst_Dist}).

We begin by estimating
\[
\bnorm{
 \big[ (1 - g_0)^{1/2} z_0 (1 - g_0)^{1/2}, \, (1 - g_1) c \big]}.
\]
Using (\ref{L_5715_CommEst_CommG0})
and the choice of $\dt_0$,
we get
\[
\bnorm{z_0 (1 - g_0)^{1/2} - (1 - g_0)^{1/2} z_0}
 < \frac{\ep}{11}.
\]
Since all the terms have norm at most~$1$
(by (\ref{L_5715_CommEst_Order}) and~(\ref{L_5715_CommEst_Norm})),
we can use this relation twice at the second step,
use (\ref{L_5715_CommEst_Commg2c})
and~(\ref{L_5715_CommEst_Commgc02}) at the third step,
and use $g_0 g_1 = g_1 g_0$
(from (\ref{L_5715_CommEst_Order})
and~(\ref{L_5715_CommEst_RelUnit})),
to get
\begin{align*}
& \bnorm{\big[ (1 - g_0)^{1/2} z_0 (1 - g_0)^{1/2}, \, (1 - g_1) c \big]}
\\
& \alignInd \mbox{}
 = \bnorm{ (1 - g_0)^{1/2} z_0 (1 - g_0)^{1/2} (1 - g_1) c
    - (1 - g_1) c (1 - g_0)^{1/2} z_0 (1 - g_0)^{1/2} }
\\
& \alignInd \mbox{}
 < \frac{2 \ep}{11}
   + \bnorm{ (1 - g_0) z_0 (1 - g_1) c
    - (1 - g_1) c (1 - g_0) z_0 }
\\
& \alignInd \mbox{}
 < \frac{2 \ep}{11} + \dt + \dt
   + \bnorm{ (1 - g_0) (1 - g_1) c z_0
    - (1 - g_1) (1 - g_0) c z_0 }
\\
& \alignInd \mbox{}
 = \frac{2 \ep}{11} + \dt + \dt
 \leq \frac{4 \ep}{11}.
\end{align*}
Since all the terms have norm at most~$1$,
it now follows from~(\ref{L_5715_CommEst_Dist})
that
\begin{equation}\label{Eq_5715_zcommcg2}
\norm{[ z, \, (1 - g_1) c ]}
 < \frac{4 \ep}{11} + 2 \bnorm{ z - (1 - g_0)^{1/2} z_0 (1 - g_0)^{1/2} }
 < \frac{4 \ep}{11} + 2 \dt
 \leq \frac{6 \ep}{11}.
\end{equation}

Next, we estimate
$\norm{z - (1 - g_1) z}$ and $\norm{z - z (1 - g_1)}$.
The relation
\[
(1 - g_1) (1 - g_0) = (1 - g_0) (1 - g_1) = 1 - g_0
\]
(from~(\ref{L_5715_CommEst_RelUnit}))
implies that
\[
(1 - g_1) (1 - g_0)^{1/2} z_0 (1 - g_0)^{1/2}
 = (1 - g_0)^{1/2} z_0 (1 - g_0)^{1/2}.
\]
and
\[
(1 - g_0)^{1/2} z_0 (1 - g_0)^{1/2} (1 - g_1)
 = (1 - g_0)^{1/2} z_0 (1 - g_0)^{1/2}.
\]
Therefore,
using~(\ref{L_5715_CommEst_Dist}) and $\norm{1 - g_1} \leq 1$
at the second step,
\begin{equation}\label{Eq_5715_Cutg2}
\norm{z - (1 - g_1) z}
 \leq \big( 1 + \norm{1 - g_1} \big)
   \bnorm{ z - (1 - g_0)^{1/2} z_0 (1 - g_0)^{1/2} }
 < 2 \dt.
\end{equation}
Similarly
\begin{equation}\label{Eq_5715_OtherCut}
\norm{z - z (1 - g_1)} < 2 \dt.
\end{equation}

At the second step in the following calculation,
we use (\ref{Eq_5715_OtherCut}) and $\norm{c} \leq 1$ on the first term,
(\ref{Eq_5715_zcommcg2}) on the second term,
(\ref{L_5715_CommEst_Commgc02}) and $\norm{z} \leq 1$ on the third term,
and (\ref{Eq_5715_Cutg2}) on the fourth term:
\begin{align*}
\norm{[ z, c ]}
& \leq \norm{z - z (1 - g_1)} \norm{c}
  + \bnorm{ z (1 - g_1) c - (1 - g_1) c z }
\\
& \alignInd \mbox{}
  + \bnorm{ (1 - g_1) c - c (1 - g_1) } \norm{z}
  + \norm{c} \norm{(1 - g_1) z - z}
\\
& < 2 \dt + \frac{6 \ep}{11} + \dt + 2 \dt
  \leq \ep.
\end{align*}
This completes the proof.
\end{proof}

Lemma~\ref{L_6207_cpcZPtb} below
is closely related to Lemma 1.2.5 of~\cite{WinCovDim2}.
The proof given in~\cite{WinCovDim2} is very sketchy.
In particular,
the elements in the range of $\ph$
are not in the linear span of the images of
the generators of $C M_n$ used there,
only in the \ca{} they generate.
We address this issue by using a different presentation of $C M_n$.
The following notation is convenient.

\begin{ntn}\label{N_6207_ConeNtn}
For $n \in \N$,
let $(e_{j, k})_{j, k = 1, 2, \ldots, n}$
be the standard system of matrix units for~$M_n$.
We take the cone $C M_n$ over $M_n$ to be
\[
C M_n
  = C_0 ((0, 1]) \otimes M_n
  = \big\{ f \in C ([0, 1], \, M_n) \colon f (0) = 0 \big\}.
\]
We let $t \in C_0 ((0, 1])$
be the function $t (\ld) = \ld$ for $\ld \in (0, 1]$.
For $j, k = 1, 2, \ldots, n$
we define $f_{j, k} \in C M_n$
by $f_{j, k} = t \otimes e_{j, k}$.

If $A$ is a \ca,
we denote its unitization by~$A^{+}$,
adding a new identity even if $A$ is already unital.
\end{ntn}

\begin{lem}\label{L_6207_ConeRel}
Let $n \in \N$.
Let $C$ be the universal \ca{}
on generators $x_{j, k}$ for $j, k = 1, 2, \ldots, n$,
subject to the following relations:
\begin{enumerate}
\item\label{L_6207_ConeRel_CMid}
$x_{j, k} x_{k, m} = x_{j, l} x_{l, m}$
for $j, k, l, m = 1, 2, \ldots, n$.
\item\label{L_6207_ConeRel_Orth}
$x_{j, j} x_{k, k} = 0$
for $j, k = 1, 2, \ldots, n$
with $j \neq k$.
\item\label{L_6207_ConeRel_Adj}
$x_{j, k} = x_{k, j}^*$ for $j, k = 1, 2, \ldots, n$.
\item\label{L_6207_ConeRel_Norm}
$\| x_{j, j} \| \leq 1$ for $j = 1, 2, \ldots, n$.
\item\label{L_6207_ConeRel_Pos}
$\| 1 - x_{j, j} \| \leq 1$ for $j = 1, 2, \ldots, n$.
\setcounter{TmpEnumi}{\value{enumi}}
\end{enumerate}
Then there is an isomorphism $\rh \colon C \to (C M_n)^{+}$
such that $\rh (x_{j, k}) = f_{j, k}$ for $j, k = 1, 2, \ldots, n$.
\end{lem}

\begin{proof}
We claim that the relations in the statement imply the following
additional relations:
\begin{enumerate}
\setcounter{enumi}{\value{TmpEnumi}}
\item\label{L_6207_ConeRel_IsP}
$x_{j, j} \geq 0$
for $j = 1, 2, \ldots, n$.
\item\label{L_6207_ConeRel_FullNorm}
$\| x_{j, k} \| \leq 1$ for $j, k = 1, 2, \ldots, n$.
\item\label{L_6207_ConeRel_AppId}
$x_{j, k} = \lim_{n \to \infty} x_{j, j}^{1/n} x_{j, k}
  = \lim_{n \to \infty} x_{j, k} x_{k, k}^{1/n}$
for $j, k = 1, 2, \ldots, n$.
\item\label{L_6207_ConeRel_FullOrth}
$x_{j, k} x_{l, m} = 0$
for $j, k, l, m = 1, 2, \ldots, n$
with $k \neq l$.
\setcounter{TmpEnumi}{\value{enumi}}
\end{enumerate}
Relation~(\ref{L_6207_ConeRel_IsP})
follows from~(\ref{L_6207_ConeRel_Adj})
(which implies that $x_{j, j}$ is selfadjoint),
(\ref{L_6207_ConeRel_Norm}), and~(\ref{L_6207_ConeRel_Pos}).
To prove~(\ref{L_6207_ConeRel_FullNorm}),
use (\ref{L_6207_ConeRel_Adj}),
(\ref{L_6207_ConeRel_CMid}), and~(\ref{L_6207_ConeRel_Norm})
to see that
$\| x_{j, k}^* x_{j, k} \| = \| x_{k, k}^2 \| \leq 1$.
For the first part of~(\ref{L_6207_ConeRel_AppId}),
use (\ref{L_6207_ConeRel_CMid}), (\ref{L_6207_ConeRel_Adj}),
and~(\ref{L_6207_ConeRel_IsP}),
to get $x_{j, j} = (x_{j, k} x_{j, k}^*)^{1/2}$;
now use the general fact $\lim_{n \to \infty} (a a^*)^{1/n} a = a$
for any element $a$ of any \ca.
The second part of~(\ref{L_6207_ConeRel_AppId}) follows by taking
adjoints and using~(\ref{L_6207_ConeRel_Adj}).
To prove~(\ref{L_6207_ConeRel_FullOrth}),
we now use (\ref{L_6207_ConeRel_AppId}) at the first step
and (\ref{L_6207_ConeRel_Orth}) at the second step
to get
\[
x_{j, k} x_{l, m}
 = \lim_{n \to \infty} x_{j, k} x_{k, k}^{1/n} x_{l, l}^{1/n} x_{l, m}
 = \lim_{n \to \infty} x_{j, k} \cdot 0 \cdot x_{l, m}
 = 0.
\]
This completes the proof of the claim.

In the rest of the proof, we follow Notation~\ref{N_6207_ConeNtn}.

It is immediate that there is a unital homomorphism
$\rh \colon C \to (C M_n)^{+}$
such that $\ph (x_{j, k}) = f_{j, k}$ for $j, k = 1, 2, \ldots, n$.
If $n = 1$
then (\ref{L_6207_ConeRel_CMid}) and~(\ref{L_6207_ConeRel_Orth})
are vacuous,
(\ref{L_6207_ConeRel_Adj}) says that $x_{1, 1}$ is selfadjoint,
(\ref{L_6207_ConeRel_Norm}) says that $\| x_{1, 1} \| \leq 1$,
and (\ref{L_6207_ConeRel_Pos})
then just says that $x_{1, 1}$ is positive.
The conclusion is now easy.
So assume $n \geq 2$.

Let $D$ be the universal \ca{}
on generators $y_{j}$ for $j = 2, 3, \ldots, n$,
subject to the relations:
\begin{enumerate}
\setcounter{enumi}{\value{TmpEnumi}}
\setcounter{enumi}{\value{TmpEnumi}}
\item\label{L_6207_SmConeRel_Diag}
$y_j^* y_j = y_2^* y_2$ for $j = 3, 4, \ldots, n$.
\item\label{L_6207_SmConeRel_PZ}
$y_j y_k = 0$ for $j, k = 2, 3, \ldots, n$.
\item\label{L_6207_SmConeRel_AdjZ}
$y_j^* y_k = 0$ for $j, k = 2, 3, \ldots, n$ with $j \neq k$.
\item\label{L_6207_SmConeRel_Norm}
$\| y_j \| \leq 1$ for $j = 2, 3, \ldots, n$.
\end{enumerate}
Recall from Proposition 3.3.1 of~\cite{Lr}
that there is an isomorphism
$\sm \colon C M_n \to D$ such that
$\sm (f_{j, 1}) = y_j$ for $j = 2, 3, \ldots, n$.
Examining the relations
(\ref{L_6207_ConeRel_CMid}),
(\ref{L_6207_ConeRel_Adj}),
(\ref{L_6207_ConeRel_FullNorm}),
and (\ref{L_6207_ConeRel_FullOrth}),
one sees that there is a unital homomorphism
$\mu \colon (C M_n)^{+} \to C$
such that $\mu (f_{j, 1}) = x_{j, 1}$
for $j = 2, 3, \ldots, n$.
Clearly $\rh \circ \mu$ is the identity map on $(C M_n)^{+}$.
The lemma will thus be proved if we show that $\mu$ is surjective.

Define $B = \mu ( (C M_n)^{+} ) \subset C$.
Obviously $1 \in B$,
and $x_{j, 1} \in B$ for $j = 2, 3, \ldots, n$.
It follows from~(\ref{L_6207_ConeRel_Adj})
that $x_{1, j} \in B$ for $j = 2, 3, \ldots, n$.
For $j = 1, 2, \ldots, n$,
recall that $x_{j, j} \geq 0$ by~(\ref{L_6207_ConeRel_IsP}).
If $j \neq 1$,
the relation~(\ref{L_6207_ConeRel_CMid})
implies that $x_{j, j}^2 = x_{j, 1} x_{1, j}$.
Thus $x_{j, j}^2 \in B$,
whence $x_{j, j} = ( x_{j, j}^2 )^{1/2} \in B$.
For $j = 1$
we use~(\ref{L_6207_ConeRel_CMid})
to get instead $x_{1, 1}^2 = x_{1, 2} x_{2, 1} \in B$,
whence $x_{1, 1} \in B$ as before.

Now let $j, k \in \{ 1, 2, \ldots, n \}$ be arbitrary.
For $r \in \N$
we use~(\ref{L_6207_ConeRel_CMid}) to get
\begin{equation}\label{Eq_6209_PowerM}
x_{j, j}^r x_{j, k} = x_{j, j}^{r - 1} x_{j, 1} x_{1, k} \in B.
\end{equation}
For $\af \in (0, \infty)$
we can approximate $\ld \mapsto \ld^{\af}$
on $[0, 1]$ by polynomials with no constant term,
and therefore deduce from~(\ref{Eq_6209_PowerM})
that $x_{j, j}^{\af} x_{j, k} \in B$.
Now $x_{j, k} \in B$ follows
from~(\ref{L_6207_ConeRel_AppId}).
We have shown that $\mu$ is surjective,
and proved the lemma.
\end{proof}

We can now give a very explicit characterization of
completely positive contractive order zero maps from~$M_n$.
For completeness,
we include the already known characterization
in terms of homomorphisms from $C M_n$.

\begin{cor}\label{C_6208_EqForCPCZ}
Let $A$ be a \ca,
let $n \in \N$,
and let $\ph \colon M_n \to A$ be a linear map.
Then the following are equivalent
(using Notation~\ref{N_6207_ConeNtn}
in (\ref{C_6208_EqForCPCZ_ConeHm}),
(\ref{C_6208_EqForCPCZ_UnitRel}),
and~(\ref{C_6208_EqForCPCZ_NoURel})):
\begin{enumerate}
\item\label{C_6208_EqForCPCZ_CPCZ}
$\ph$ is completely positive contractive and has order zero.
\item\label{C_6208_EqForCPCZ_ConeHm}
There is a homomorphism
$\pi \colon C M_n \to A$
such that $\ph (z) = \pi (t \otimes z)$
for all $z \in M_n$.
\item\label{C_6208_EqForCPCZ_UnitRel}
The elements $x_{j, k} = \ph (e_{j, k})$
for $j, k = 1, 2, \ldots, n$
satisfy the relations
(\ref{L_6207_ConeRel_CMid})--(\ref{L_6207_ConeRel_Pos})
of Lemma~\ref{L_6207_ConeRel} as elements of the \ca~$A^+$.
\item\label{C_6208_EqForCPCZ_NoURel}
The elements $x_{j, k} = \ph (e_{j, k})$
for $j, k = 1, 2, \ldots, n$
satisfy the relations
(\ref{L_6207_ConeRel_CMid})--(\ref{L_6207_ConeRel_Norm})
of Lemma~\ref{L_6207_ConeRel},
together with the relation
$x_{j, j} \geq 0$ for $j = 1, 2, \ldots, n$, as elements of the \ca~$A$.
\end{enumerate}
\end{cor}

\begin{proof}
The equivalence of (\ref{C_6208_EqForCPCZ_CPCZ})
and~(\ref{C_6208_EqForCPCZ_ConeHm})
is Corollary~4.1 of~\cite{WnZc}
(or the proposition in 1.2.2 of~\cite{WinCovDim2}).
That (\ref{C_6208_EqForCPCZ_ConeHm})
implies (\ref{C_6208_EqForCPCZ_NoURel}) is clear,
as is the implication from (\ref{C_6208_EqForCPCZ_NoURel})
to (\ref{C_6208_EqForCPCZ_UnitRel}).
So assume (\ref{C_6208_EqForCPCZ_UnitRel}).
Lemma~\ref{L_6207_ConeRel}
provides a unital homomorphism
$\sm \colon (C M_n)^{+} \to A^{+}$
such that $\sm (t \otimes e_{j, k}) = \ph (e_{j, k})$
for $j, k = 1, 2, \ldots, n$.
Since $\ph (e_{j, k}) \in A$ for $j, k = 1, 2, \ldots, n$,
it follows that $\pi = \sm |_{C M_n}$
is a homomorphism from $C M_n$ to~$A$
such that $\pi (t \otimes e_{j, k}) = \ph (e_{j, k})$
for $j, k = 1, 2, \ldots, n$.
Condition~(\ref{C_6208_EqForCPCZ_ConeHm}) is now immediate.
\end{proof}

\begin{lem}\label{L_6207_cpcZPtb}
For every $\ep > 0$ there is $\dt > 0$ such that the following holds.
Let $A$ be a \ca,
let $B \subset A$ be a subalgebra,
let $n \in \N$,
let $\ph_0 \colon M_n \to A$
be a completely positive contractive order zero map,
and let $x \in B$ satisfy:
\begin{enumerate}
\item\label{L_6207_cpcZPtb_Ino1}
$0 \leq x \leq 1$.
\item\label{L_6207_cpcZPtb_Comm}
With $e_{j,k}$ as in Notation \ref{N_6207_ConeNtn},
we have
$\norm{ [ x, \ph_0 (e_{j, k}) ] } < \dt$ for $j, k = 1, 2, \ldots, n$.
\item\label{L_6207_cpcZPtb_Dist}
$\dist \big( \ph_0 (e_{j, k}) x, \, B \big) < \dt$
for $j, k = 1, 2, \ldots, n$.
\end{enumerate}
Then there is a
completely positive contractive order zero map
$\ph \colon M_n \to B$
such that for all $z \in M_n$ with $\| z \| \leq 1$,
we have $\| \ph_0 (z) x - \ph (z) \| < \ep$.
\end{lem}

\begin{proof}
We first consider the case $n = 1$ separately,
since one step of the argument for $n \geq 2$
doesn't work in this case.

Define a continuous function $f \colon \R \to [0, 1]$
by
\[
f (\ld)
 = \begin{cases}
   0 & \ld \leq 0
        \\
   \ld & 0 \leq \ld \leq 1
       \\
   1 & 1 \leq \ld.
\end{cases}
\]
By a polynomial approximation argument,
there is $\dt_0 > 0$ such that
$\dt_0 \leq \min (1, \ep)$
and whenever
$A$ is a \ca{} and $c, d \in A_{\mathrm{sa}}$ satisfy
\[
\| c \| \leq 2,
\,\,\,\,\,\,
\| d \| \leq 2,
\andeqn
\| c - d \| < \dt_0,
\]
then $\big\| f (c) - f (d) \big\| < \frac{\ep}{2}$.
Use Lemma 2.5 of~\cite{ArPh}
to choose $\dt > 0$ such that
$\dt \leq \frac{\dt_0}{2}$
and whenever
$A$ is a \ca{} and $a, x \in A$ satisfy
\[
\| a \| \leq 1,
\,\,\,\,\,\,
\| x \| \leq 1,
\,\,\,\,\,\,
x \geq 0,
\andeqn
\| [x, a] \| < \dt,
\]
then
$\norm{ \big[ x^{1/2}, a \big] } < \frac{\dt_0}{2}$.

Now let $A$, $B$, $\ph_0 \colon \C \to A$, and $x$
be as in the hypotheses.
Then
\begin{equation}\label{Eq_6207_PhEst}
\bnorm{ \ph_0 (1) x - x^{1/2} \ph_0 (1) x^{1/2} }
 \leq \bnorm{ \big[ x^{1/2}, \, \ph_0 (1) \big] } \bnorm{ x^{1/2} }
 < \frac{\dt_0}{2}
\end{equation}
and
\[
\dist ( \ph_0 (1) x, \, B)
 < \frac{\dt_0}{2}.
\]
These inequalities allow us to choose
$d_0 \in B$ such that
$\norm{ d_0 -  x^{1/2} \ph_0 (1) x^{1/2} } < \dt_0$.
Then $d = \frac{1}{2} (d_0 + d_0^*)$
also satisfies
$\norm{ d -  x^{1/2} \ph_0 (1) x^{1/2} } < \dt_0$.
In particular, $\| d \| \leq 2$.
Since
\[
f \big( x^{1/2} \ph_0 (1) x^{1/2} \big) = x^{1/2} \ph_0 (1) x^{1/2},
\]
the choice of $\dt_0$
implies that
$\norm{ f (d) -  x^{1/2} \ph_0 (1) x^{1/2} } < \frac{\ep}{2}$.
Combining this estimate with~(\ref{Eq_6207_PhEst}) and $\dt_0 \leq \ep$
gives $\| f (d) - \ph_0 (1) x \| < \ep$.
Now we can define $\ph \colon \C \to B$
by $\ph (\ld) = \ld f (d)$ for $\ld \in \C$.

Now assume $n \geq 2$.
For $\dt_0 > 0$ consider the following relations
on elements $y_{j, k}$ in a \ca,
for $j, k = 1, 2, \ldots, n$:
\begin{enumerate}
\item\label{L_6207_SftConeRel_CMid}
$\| y_{j, k} y_{k, m} - y_{j, l} y_{l, m} \| < \dt_0$
for $j, k, l, m = 1, 2, \ldots, n$.
\item\label{L_6207_SftConeRel_Orth}
$\| y_{j, j} y_{k, k} \| < \dt_0$
for $j, k = 1, 2, \ldots, n$
with $j \neq k$.
\item\label{L_6207_SftConeRel_Adj}
$\| y_{j, k} - y_{k, j}^* \| < \dt_0$ for $j, k = 1, 2, \ldots, n$.
\item\label{L_6207_SftConeRel_Norm}
$\| y_{j, j} \| < 1 + \dt_0$ for $j = 1, 2, \ldots, n$.
\item\label{L_6207_SftConeRel_Pos}
$\| 1 - y_{j, j} \| < 1 + \dt_0$ for $j = 1, 2, \ldots, n$.
\end{enumerate}

Since the cone $C M_n$ is projective
(see Theorem 10.2.1 of~\cite{Lr}),
it is semiprojective,
so its unitization $(C M_n)^{+}$ is semiprojective
(by Theorem 14.1.7 of~\cite{Lr}).
It therefore follows from Theorem 14.1.4 of~\cite{Lr}
that the relations in Lemma~\ref{L_6207_ConeRel}
are stable in the sense of Definition 14.1.1 of~\cite{Lr}.
Thus,
there is $\dt_0 > 0$ such that,
whenever
$D$ is a unital \ca{}
and $y_{j, k}$, for $j, k = 1, 2, \ldots, n$,
are elements of $D$ satsfying
(\ref{L_6207_SftConeRel_CMid})--(\ref{L_6207_SftConeRel_Pos}) above,
then there is a unital homomorphism
$\sm \colon (C M_n)^{+} \to D$ such that,
with $f_{j, k}$ as defined in Notation~\ref{N_6207_ConeNtn},
we have
\begin{equation}\label{Eq_6207_CloseHm}
\| \sm (f_{j, k}) - y_{j, k} \| < \frac{\ep}{2 n^2}
\end{equation}
for $j, k = 1, 2, \ldots, n$.
Use Lemma 2.5 of~\cite{ArPh}
to choose $\dt > 0$ such that
\begin{equation}\label{Eq_6209_ChooseDt}
\dt
 \leq \min \left( 1, \, \frac{\dt_0}{8}, \, \frac{\ep}{2 n^2} \right)
\end{equation}
and whenever
$A$ is a \ca{} and $a, x \in A$ satisfy
\[
\| a \| \leq 1,
\,\,\,\,\,\,
\| x \| \leq 1,
\,\,\,\,\,\,
x \geq 0,
\andeqn
\| [x, a] \| < \dt,
\]
then
$\norm{ \big[ x^{1/2}, a \big] } < \frac{\dt_0}{2}$.

Now let $A$, $B$, $\ph_0 \colon M_n \to A$, and $x$
be as in the hypotheses.
For $j, k = 1, 2, \ldots, n$
choose $y_{j, k} \in B$
such that
\begin{equation}\label{Eq_6207_yDef}
\| y_{j, k} - \ph_0 (e_{j, k}) x \| < \dt.
\end{equation}
Then
$\| y_{j, k} \| < 1 + \dt$.
We claim that the relations
(\ref{L_6207_SftConeRel_CMid})--(\ref{L_6207_SftConeRel_Pos}) above
are satisfied in~$B^{+}$.
First,
by Corollary \ref{C_6208_EqForCPCZ}(\ref{C_6208_EqForCPCZ_UnitRel}),
the elements $x_{j, k} = \ph_0 (e_{j, k})$
satisfy the relations
(\ref{L_6207_ConeRel_CMid})--(\ref{L_6207_ConeRel_Pos})
of Lemma~\ref{L_6207_ConeRel} as elements of the \ca~$A^+$.
We now verify~(\ref{L_6207_SftConeRel_CMid}).
We have,
using (\ref{Eq_6207_yDef}) and
\[
\| \ph_0 (e_{j, k}) x \| \leq 1,
\,\,\,\,\,\,
\| y_{j, k} \| < 1 + \dt,
\andeqn
\| y_{j, l} \| < 1 + \dt
\]
at the first step,
\begin{align*}
& \| y_{j, k} y_{k, m} - y_{j, l} y_{l, m} \|
\\
& \hspace*{3em} {\mbox{}}
  < 2 \dt (1 + \dt) + 2 \dt
      + \big\| \ph_0 (e_{j, k}) x \ph_0 (e_{k, m}) x
              -\ph_0 (e_{j, l}) x \ph_0 (e_{l, m}) x \big\|
\\
& \hspace*{3em} {\mbox{}}
  \leq 2 \dt (2 + \dt)
    + \| [ \ph_0 (e_{k, m}), x] \| + \| [ \ph_0 (e_{l, m}), x] \|
\\
& \hspace*{6em} {\mbox{}}
    + \big\| \ph_0 (e_{j, k}) \ph_0 (e_{k, m}) x^2
      - \ph_0 (e_{j, l}) \ph_0 (e_{l, m}) x^2 \big\|.
\end{align*}
The last term in the last expression is zero
by condition~(\ref{L_6207_ConeRel_CMid}) in Lemma~\ref{L_6207_ConeRel},
and,
using hypothesis~(\ref{L_6207_cpcZPtb_Comm}) at the first step
and~(\ref{Eq_6209_ChooseDt}) at the second and third steps,
\[
2 \dt (2 + \dt)
   + \| [ \ph_0 (e_{k, m}), x] \| + \| [ \ph_0 (e_{l, m}), x] \|
 < 2 \dt (2 + \dt)  + 2 \dt
 \leq 8 \dt
 \leq \dt_0.
\]
Thus~(\ref{L_6207_SftConeRel_CMid}) holds.
Similarly,
using, in order,
(\ref{L_6207_ConeRel_Orth}),
(\ref{L_6207_ConeRel_Adj}),
and~(\ref{L_6207_ConeRel_Norm})
in Lemma~\ref{L_6207_ConeRel},
for $j, k = 1, 2, \ldots, n$
with $j \neq k$
we have
\begin{align*}
\| y_{j, j} y_{k, k} \|
& < \dt (1 + \dt) + \dt + \| \ph_0 (e_{j, j}) x \ph_0 (e_{k, k}) x \|
 \\
& < \dt (1 + \dt) + \dt + \dt
     + \bnorm{ \ph_0 (e_{j, j}) \ph_0 (e_{k, k}) x^2 }
  = \dt (3 + \dt)
  \leq 4 \dt
  < \dt_0,
\end{align*}
for $j, k = 1, 2, \ldots, n$
we have
\[
\| y_{j, k} - y_{k, j}^* \|
  < 2 \dt + \| \ph_0 (e_{j, k}) x - x \ph_0 (e_{k, j})^* \|
  < 3 \dt + \| \ph_0 (e_{j, k}) x - \ph_0 (e_{k, j})^* x \|
  < \dt_0,
\]
and for $j = 1, 2, \ldots, n$
we have
\[
\| y_{j, j} \|
 < \dt + \| \ph_0 (e_{j, j}) x \|
 \leq 1 + \dt
 < 1 + \dt_0.
\]
Finally, for $j = 1, 2, \ldots, n$,
the choice of $\dt$ and hypothesis~(\ref{L_6207_cpcZPtb_Comm})
imply that
$\bnorm{ \big[ x^{1/2}, \, \ph_0 (x_{j, j}) \big] } < \frac{\dt_0}{2}$,
so
\begin{align*}
\bnorm{ y_{j, j} - x^{1/2} \ph_0 (x_{j, j}) x^{1/2} }
& \leq \norm{ y_{j, j} - \ph_0 (x_{j, j}) x }
    + \bnorm{ \big[ x^{1/2}, \, \ph_0 (x_{j, j}) \big] }
            \bnorm{ x^{1/2} }
\\
& < \dt + \frac{\dt_0}{2}
  \leq \dt_0.
\end{align*}
Since $0 \leq x^{1/2} \ph_0 (x_{j, j}) x^{1/2} \leq 1$,
we have
\[
\norm {y_{j, j} } \leq 1 + \dt_0
\andeqn
\bnorm{ 1 - x^{1/2} \ph_0 (x_{j, j}) x^{1/2} } \leq 1.
\]
Therefore
$\| 1 - y_{j, j} \| < 1 + \dt_0$.
This completes the verification of
(\ref{L_6207_SftConeRel_CMid})--(\ref{L_6207_SftConeRel_Pos}).

By the choice of~$\dt_0$,
there is a unital homomorphism
$\sm \colon (C M_n)^{+} \to B^{+}$ such that
(\ref{Eq_6207_CloseHm}) holds
for $j, k = 1, 2, \ldots, n$.
Since $n \geq 2$,
there are no nonzero homomorphisms
from $C M_n$ to~$\C$.
It follows that
the formula $\ph (z) = \sm (t \otimes z)$,
for $z \in M_n$,
defines a completely positive contractive order zero map
from $M_n$ to~$B$.
For $j, k = 1, 2, \ldots, n$,
we have,
using (\ref{Eq_6207_CloseHm})
and~(\ref{Eq_6207_yDef}) at the second step,
and~(\ref{Eq_6209_ChooseDt}) at the third step,
\[
\norm{ \ph (e_{j, k}) - \ph_0 (e_{j, k}) x }
  \leq \norm{ \ph (e_{j, k}) - y_{j, k} }
       + \norm{ y_{j, k} - \ph_0 (e_{j, k}) x }
  < \frac{\ep}{2 n^2} + \dt
  \leq \frac{\ep}{n^2}.
\]
Now let $z \in M_n$ satisfy $\| z \| \leq 1$.
Choose $\ld_{j, k} \in \C$
for $j, k = 1, 2, \ldots, n$
such that $z = \sum_{j, k = 1}^n \ld_{j, k} e_{j, k}$.
One easily sees that $| \ld_{j, k} | \leq 1$
for $j, k = 1, 2, \ldots, n$.
So
\[
\| \ph (z) - \ph_0 (z) x \|
  \leq \sum_{j, k = 1}^n
     | \ld_{j, k} | \norm{ \ph (e_{j, k}) - \ph_0 (e_{j, k}) x }
  < n^2 \left( \frac{\ep}{n^2} \right)
  = \ep.
\]
This completes the proof.
\end{proof}

For unital \ca{s},
we strengthen Lemma~\ref{L_6207_cpcZPtb}
by adding a condition to the conclusion,
as follows.

\begin{lem}\label{L_5714_MainPerturb}
For every $\ep > 0$ and $n \in \N$, there is $\dt > 0$ such
that the following holds.
Whenever $A$, $B$, $\ph_{0} \colon M_{n} \to A$, and $x \in B$ satisfy
the conditions in Lemma~\ref{L_6207_cpcZPtb},
and in addition $A$ is unital and $B$ contains the identity of~$A$,
there exists a completely positive contractive order zero map
$\ph \colon M_{n} \to B$
such that:
\begin{enumerate}
\item\label{L_5714_MainPerturb_Dist}
$\norm{\ph (z) - \ph_0 (z) x} < \ep$ for all $z \in M_{n}$ with
$\norm{z} \leq 1$.
\item\label{L_5714_MainPerturb_Cuntz}
$1 - \ph (1) \cuntz_A (1 - x) \oplus [1 - \ph_0 (1)]$.
\end{enumerate}
\end{lem}

\begin{proof}
Set $\ep_0 = \min \big( \frac{1}{3}, \, \frac{\ep}{3} \big)$.
Use Lemma 2.5 of~\cite{ArPh}
to choose $\dt_0 > 0$ such that whenever
$A$ is a \ca{} and $a, x \in A$ satisfy
\[
\norm{a} \leq 1,
\,\,\,\,\,\,
\norm{x} \leq 1,
\,\,\,\,\,\,
x \geq 0,
\andeqn
\norm{ [x, a] } < \dt_0,
\]
then
$\norm{ \big[ x^{1/2}, a \big] } < \ep_0$.
Apply Lemma~\ref{L_6207_cpcZPtb}
with $\min (\ep_0, \dt_0)$ in place of~$\ep$,
getting $\dt_1 > 0$.
Set $\dt = \min (\dt_0, \dt_1)$.

Now let $A$ be a \uca,
and let $B \subset A$,
$\ph_{0} \colon M_{n} \to A$, and $x \in B$ be as in the hypotheses.
The choice of $\dt_1$ using Lemma~\ref{L_6207_cpcZPtb}
gives us a
completely positive contractive order zero map
$\ph_1 \colon M_{n} \to B$
such that
\begin{equation}\label{Eq_6208_Conq}
\norm{\ph_1 (z) - \ph_0 (z) x} < \min (\ep_0, \dt_0)
\end{equation}
for all $z \in M_{n}$ with
$\norm{z} \leq 1$.
The conditions on $\ph_0$
imply that $\norm{ [\ph_{0} (1), \, x] } < \dt_0$,
so by the choice of
$\dt_0$ we have $\bnorm{ \big[ \ph_{0} (1), \, x^{1/2} \big] } < \ep_0$.
Combining this estimate with
the case $z = 1$ of~(\ref{Eq_6208_Conq}),
we get
\begin{equation}\label{Eq_5714_half}
\bnorm{ \ph_1 (1) - x^{1/2} \ph_0 (1) x^{1/2}} < 2 \ep_0.
\end{equation}

Define a continuous function $f \colon [0, 1] \to [0, 1]$
by
\[
f (\ld)
 = \begin{cases}
   (1 - 2 \ep_0)^{-1} \ld & 0 \leq \ld \leq 1 - 2 \ep_0
       \\
   1                      & 1 - 2 \ep_0 \leq \ld \leq 1.
\end{cases}
\]
Following the functional calculus for
completely positive order zero maps
in Corollary~4.2 of~\cite{WnZc},
define a completely positive contractive order zero map
$\ph \colon M_{n} \to B$
by $\ph = f (\ph_1)$.

We verify part~(\ref{L_5714_MainPerturb_Dist}) of the conclusion.
Let $C \subset B$ be the \ca{} generated by $\ph_1 (M_{n})$.
Theorem~3.3 of~\cite{WnZc}
provides a homomorphism $\pi \colon M_n \to M (C)$
(the multiplier algebra of~$C$)
whose range commutes with $\ph_1 (1)$
and such that
\begin{equation}\label{Eq_5713_ph1zz}
\ph_1 (z) = \pi (z) \ph_1 (1)
\end{equation}
for all $z \in M_{n}$.
Therefore $\pi (1) \ph_1 (1) = \ph_1 (1)$.
For any continuous function $g \colon [0, \infty) \to \C$
with $g (0) = 0$,
approximation by polynomials with no constant term gives
\begin{equation}\label{Eq_5713_UnitForph1}
g (\ph_1 (1)) = \pi (1) g (\ph_1 (1)).
\end{equation}
By definition
(see Corollary~4.2 of~\cite{WnZc}),
we have
\begin{equation}\label{Eq_5713_phOfz}
\ph (z) = \pi (z) f (\ph_1 (1) )
\end{equation}
for all $z \in M_{n}$.
In particular,
$\ph (1) = \pi (1) f (\ph_1 (1) )$,
so~(\ref{Eq_5713_UnitForph1})
implies that
$\ph (1) = f (\ph_1 (1) )$.
Since $| f (\ld) - \ld | \leq 2 \ep_0$ for all $\ld \in [0, 1]$,
we have
$\norm{\ph_1 (1) - f (\ph_1 (1) )} \leq 2 \ep_0$.
Combining this estimate with
(\ref{Eq_5713_ph1zz}) and~(\ref{Eq_5713_phOfz})
gives
$\norm{\ph (z) - \ph_1 (z)} \leq 2 \ep_0 \norm{z}$
for all $z \in M_{n}$.
Using~(\ref{Eq_6208_Conq}),
we now get
$\norm{\ph (z) - \ph_0 (z) x} < 3 \ep_0 \leq \ep$
for all $z \in M_{n}$ with
$\norm{z} \leq 1$,
as desired.

It remains to verify part~(\ref{L_5714_MainPerturb_Cuntz})
of the conclusion.
For $\ld \in [0, 1]$,
we have
\[
1 - f (\ld)
 = (1 - 2 \ep_0)^{-1} \max (0, \, 1 - \ld - 2 \ep_0).
\]
Since $\ph (1) = f (\ph_1 (1)$, it follows that
\begin{equation}\label{Eq_5714_Sub2ep0}
1 - \ph (1)
 \sim_A ( 1 - \ph_1 (1) - 2 \ep_0 )_{+}.
\end{equation}
By~(\ref{Eq_5714_half}),
we have
\[
\bnorm{ [1 - \ph_1 (1)] - \big[ 1 - x^{1/2} \ph_0 (1) x^{1/2} \big] }
 < 2 \ep_0,
\]
whence
\begin{equation}\label{Eq_5714_ToHalf}
( 1 - \ph (1) - 2 \ep_0 )_{+} \cuntz_A 1 - x^{1/2} \ph_0 (1) x^{1/2}.
\end{equation}
Moreover,
using Lemma 1.4(4) of~\cite{PhLarge}
at the third step,
we get
\begin{align*}
1 - x^{1/2} \ph_0 (1) x^{1/2}
& = 1 - x + x^{1/2} [1 - \ph_0 (1) ] x^{1/2}
\\
& \cuntz_A (1 - x) \oplus x^{1/2} [1 - \ph_0 (1) ] x^{1/2}
\\
& \sim_A (1 - x) \oplus [1 - \ph_0 (1) ]^{1/2} x [1 - \ph_0 (1) ]^{1/2}
\\
& \leq (1 - x) \oplus [1 - \ph_0 (1) ].
\end{align*}
Combining this result with
(\ref{Eq_5714_Sub2ep0}) and (\ref{Eq_5714_ToHalf}) gives
\[
1 - \ph (1) \cuntz_A (1 - x) \oplus [1 - \ph_0 (1) ],
\]
as desired.
\end{proof}

\section{Tracial $\JS$-Absorption}\label{Sec_TZAbs}

In this section, we prove our main results.

The following definition first appeared in~\cite{HirOr2013}.

\begin{dfn}[Definition~2.1 of~\cite{HirOr2013}]\label{D_TracialZAbs}
Let $A$ be a unital \ca.
We say that $A$ is
{\emph{tracially $\JS$-absorbing}} if $A \not \cong \C$
and for any $\ep > 0$, any finite set $F \subset A$,
any $n \in \N$, and any $x \in A_{+} \setminus \set{0}$,
there is a completely positive contractive order zero map
$\ph \colon M_{n} \to A$ such that:
\begin{enumerate}
\item\label{D_TracialZAbs_Small}
$1 - \ph (1) \cuntz_{A} x$.
\item\label{D_TracialZAbs_AppCm}
For any $y \in M_{n}$ with $\norm{y} = 1$ and any $a \in F$,
we have $\norm{\ph (y) a - a \ph (y) } < \ep$.
\end{enumerate}
\end{dfn}

\begin{thm}\label{tZabsGoesUp}
Let $A$ be a simple infinite dimensional \uca,
and let $B$ be a centrally large subalgebra of~$A$.
If $B$ is tracially $\JS$-absorbing,
then $A$ is tracially $\JS$-absorbing.
\end{thm}

We don't need to assume that $A$ is finite.
Theorem~3.3 of~\cite{HirOr2013}
shows that if $B$ is tracially $\JS$-absorbing,
then $B$ has strict comparison of positive elements,
from which it follows
that $B$ either has a normalized quasitrace or is purely infinite.
In the first case,
$B$ is finite,
so $A$ is also finite
(Proposition~6.15 of~\cite{PhLarge}).
In the second case,
every purely infinite simple unital \ca{}
is tracially $\JS$-absorbing,
as one sees by taking $\ph = 0$
in Definition~\ref{D_TracialZAbs},
and if $B$ is purely infinite then so is~$A$
by Proposition~6.17 of~\cite{PhLarge}.

\begin{proof}[Proof of Theorem~\ref{tZabsGoesUp}]
We verify the conditions in Definition~\ref{D_TracialZAbs}.
So let $\ep > 0$, let $F \subset A$ be a finite set,
let $x \in A_{+} \setminus \{ 0 \}$, and let $n \in \N$.
Write $F = \{ a_{1}, a_2, \ldots, a_{m} \}$.
Without loss of generality,
we may assume that $\norm{a_{j}} \leq 1$ for $j = 1, 2, \ldots, m$.
By Lemma~2.4 of~\cite{PhLarge},
there exist $x_{1}, x_{2} \in A_{+} \setminus \set{0}$ such that
\begin{equation}\label{Eq_5715_Getx1x2}
x_{1} \sim x_{2},
\,\,\,\,\,\,
x_{1} x_{2} = 0,
\andeqn
x_{1} + x_{2} \in {\overline{x A x}}.
\end{equation}
We may clearly assume that $\norm{x_{1}} = \norm{x_{2}} = 1$.

Apply Lemma~\ref{L-phZcjCom} with $\frac{\ep}{3}$ in
place of $\ep$,
obtaining $\dt_0 > 0$.
Set
$\dt_1 = \min \big( \frac{\ep}{3}, \frac{\dt_0}{3} \big) > 0$.
Apply Lemma~\ref{L_9430_LgTower}
with $\dt_1$ in place of~$\ep$,
with $a_{1}, a_2, \ldots, a_{m}$ as given,
with $x_{1} \in A_{+} \setminus \set{0}$ in place of~$x$,
with $1_{A} \in B_{+} \setminus \set{0}$ in place of~$y$,
and with $N = 1$,
to obtain $c_{1}, c_2,  \ldots, c_{m} \in A$
and $g_{0}, g_{1} \in B$
such that:
\begin{enumerate}
\item\label{tZabsGoesUp_1}
$0 \leq g_{1} \leq g_0 \leq 1$
and $g_{0} g_{1} = g_{1}$.
\item\label{tZabsGoesUp_2}
$\norm{c_{j}} \leq \norm{a_{j}}$ for $j = 1, 2, \ldots, m$.
\item\label{tZabsGoesUp_3}
$\norm{c_{j} - a_{j}} < \dt_1$ for $j = 1, 2, \ldots, m$.
\item\label{tZabsGoesUp_4}
$(1 - g_{0}) c_{j}, \, (1 - g_{1}) c_{j} \in B$ for
$j = 1, 2, \ldots, m$.
\item\label{tZabsGoesUp_5}
$g_{k} \cuntz_{B} 1_{A}$ and $g_{k} \cuntz_{A} x_{1}$ for
$k = 0, 1$.
\item\label{tZabsGoesUp_7}
$\norm{ g_{0} a_{j} - a_{j} g_{0} } < \dt_1$
and $\norm{ g_{1} a_{j} - a_{j} g_{1} } < \dt_1$
for $j = 1, 2, \ldots, m$.
\setcounter{TmpEnumi}{\value{enumi}}
\end{enumerate}

Apply Lemma 5.3 of~\cite{PhLarge}
with $r = 1_{A}$
to obtain $b \in B_{+} \setminus \set{0}$
satisfying
\begin{equation}\label{Eq_5715_Getb}
b \cuntz_{A} x_{2}.
\end{equation}
(We do not use condition (3) of that lemma,
and we only use
condition (1) of that lemma
to guarantee that $b \neq 0$.)
Apply
Lemma~\ref{L_5714_MainPerturb}
with $n$ as given and
with $\dt_1$
in place of~$\ep$;
let $\dt_2 > 0$
be the resulting strictly positive number.
Set $\dt_3 = \min (\dt_1, \dt_2) > 0$ and
\[
S = \bset{g_{0}, \, 1 - g_{0}, \, (1 - g_{0})^{1/2} }
 \cup \bset{ (1 - g_{1}) c_{j} \colon j = 1, 2, \ldots, m} \subset B.
\]
Since $B$ is tracially $\JS$-absorbing,
there is a completely positive contractive order zero map
$\ph_{0} \colon M_{n} \to B$ such that:
\begin{enumerate}
\setcounter{enumi}{\value{TmpEnumi}}
\item\label{TZA_1}
$1_{A} - \ph_{0}(1) \cuntz_{B} b$.
\item\label{TZA_2}
$\norm{ [\ph_{0} (z), \, y] } < \dt_3$
for all $z \in M_{n}$ with $\norm{z} \leq 1$
and all $y \in S$.
\end{enumerate}

For $z \in M_n$ with $\| z \| \leq 1$,
we have $\norm{ [\ph_{0} (z), \, 1 - g_{0}] } < \dt_3 \leq \dt_2$.
We apply the choice of $\dt_2$ using Lemma~\ref{L_5714_MainPerturb}
with $1 - g_{0}$ in place of~$x$
and taking $A$ and $B$ there to
be both equal to~$B$.
We get a completely positive contractive order zero
map $\ph \colon M_{n} \to B$
(which we regard as a map to~$A$)
such that
$\norm{ \ph (z) - \ph_{0} (z) (1 - g_{0}) } < \dt_1$
for all $z \in M_{n}$ with $\norm{z} \leq 1$
and such that
\begin{equation}\label{Eq_5715_Domg0b}
1 - \ph (1) \cuntz_B g_0 \oplus [1 - \ph_0 (1)].
\end{equation}
For any such $z$, we compute,
using~(\ref{TZA_2}) and the definition of~$S$ at the third step,
\begin{align*}
& \bnorm{ \ph (z) - (1 - g_{0})^{1/2} \ph_{0} (z) (1 - g_{0})^{1/2} }
\\
& \hspace*{3em} {\mbox{}}
 \leq \norm{ \ph (z) - \ph_{0} (z) (1 - g_{0}) }
  + \bnorm{ \ph_{0} (z) (1 - g_{0}) - (1 - g_{0})^{1/2} \ph_{0} (z)
         (1 - g_{0})^{1/2} }
\\
& \hspace*{3em} {\mbox{}}
 < \dt_1 + \bnorm{ \big[ \ph_{0} (z), \, (1 - g_{0})^{1/2} \big] }
     \bnorm{ (1 - g_{0})^{1/2} } \\
& \hspace*{3em} {\mbox{}}
 < \dt_1 + \dt_3
  \leq 2 \dt_1
  \leq \dt_0.
\end{align*}

For $j = 1, 2, \ldots, m$,
we want to
apply the choice of $\dt_0$ using Lemma~\ref{L-phZcjCom}
with $c_j$ in place of~$c$,
with $\ph (z)$ in place of~$z$,
and with $\ph_0 (z)$ in place of~$z_0$.
We have just verified condition~(\ref{L_5715_CommEst_Dist})
of Lemma~\ref{L-phZcjCom}.
Conditions (\ref{L_5715_CommEst_Order}),
(\ref{L_5715_CommEst_RelUnit}),
(\ref{L_5715_CommEst_Norm}),
(\ref{L_5715_CommEst_Commg2c}),
and~(\ref{L_5715_CommEst_CommG0}) of Lemma~\ref{L-phZcjCom}
follow from the requirements $\| a_j \| \leq 1$,
$\| z \| \leq 1$,
$\dt_1 \leq \dt_0$,
the choice of~$S$,
and (\ref{tZabsGoesUp_1}),
(\ref{tZabsGoesUp_2}),
and (\ref{TZA_2}) above.
It remains to verify condition~(\ref{L_5715_CommEst_Commgc02})
of Lemma~\ref{L-phZcjCom}.
Using (\ref{tZabsGoesUp_1}),
(\ref{tZabsGoesUp_3}),
and (\ref{tZabsGoesUp_7}) above,
for $k = 0, 1$ we get
\[
\| [c_j, g_k] \|
 \leq 2 \| c_j - a_j \| \| g_k \| + \| [a_j, g_k] \|
 < 2 \dt_1 + \dt_1
 \leq \dt_0,
\]
as desired.
The choice of $\dt_0$ using Lemma~\ref{L-phZcjCom}
therefore implies $\norm{ [\ph (z), \, c_{j}] } < \frac{\ep}{3}$.
We now estimate
\[
\norm{ [\ph (z), \, a_{j}] }
 \leq 2 \norm{a_{j} - c_{j}} \norm{ \ph (z) }
    + \norm{ [\ph (z), \, c_{j}] }
  < 2 \dt_1 + \frac{\ep}{3}
  \leq \ep.
\]
We have verified condition~(\ref{D_TracialZAbs_AppCm})
in Definition~\ref{D_TracialZAbs}.

It remains to show that $1_{A} - \ph (1) \cuntz_{A} x$.
Using (\ref{Eq_5715_Domg0b}) at the first step,
using (\ref{tZabsGoesUp_5}) and (\ref{TZA_1}) at the second step,
using (\ref{Eq_5715_Getb}) at the third step,
and using (\ref{Eq_5715_Getx1x2}) at the fourth step,
we get
\[
1_{A} - \ph (1)
 \cuntz_{B} g_0 \oplus [1 - \ph_0 (1)]
 \cuntz_{A} x_1 \oplus b
 \cuntz_{A} x_1 \oplus x_2
 \cuntz_{A} x,
\]
as required.
\end{proof}

\begin{thm}\label{T_5715_ZabsGoesUp}
Let $A$ be a simple separable infinite dimensional nuclear \uca,
and let $B$ be a centrally large subalgebra of~$A$.
If $\JS \otimes B \cong B$ then $\JS \otimes A \cong A$.
\end{thm}

\begin{proof}
Proposition~2.2 of~\cite{HirOr2013}
implies that $B$ is tracially $\JS$-absorbing.
So $A$ is tracially $\JS$-absorbing
by Theorem~\ref{tZabsGoesUp}.
Since $A$ is nuclear,
Theorem~4.1 of~\cite{HirOr2013}
implies that $\JS \otimes A \cong A$.
\end{proof}

\begin{thm}\label{T_6208_ZDown}
Let $A$ be a stably finite simple infinite dimensional \uca,
and let $B$ be a centrally large subalgebra of~$A$.
If $A$ is tracially $\JS$-absorbing,
then $B$ is tracially $\JS$-absorbing.
\end{thm}

As will be clear from the proof,
we can drop the stable finiteness requirement
if we require that $B$ be stably centrally large in~$A$.

\begin{proof}[Proof of Theorem~\ref{T_6208_ZDown}]
We verify the conditions in Definition~\ref{D_TracialZAbs}.
So let $\ep > 0$, let $F \subset B$ be a finite set,
let $x \in B_{+} \setminus \{ 0 \}$, and let $n \in \N$.
Without loss of generality,
we may assume that
$\norm{a} \leq 1$ for all $a \in F$.
By Lemma~2.4 of~\cite{PhLarge},
there exist $x_{1}, x_{2}, x_3 \in B_{+} \setminus \set{0}$ such that
\begin{equation}\label{Eq_6208_Getx1x2x3}
x_{1} \sim x_{2} \sim x_3,
\,\,\,\,\,\,
x_{1} x_{2} = x_1 x_3 = x_2 x_3 = 0,
\andeqn
x_{1} + x_{2} + x_3 \in {\overline{x B x}}.
\end{equation}
Apply Lemma~\ref{L_5714_MainPerturb}
with $n$ as given and
with $\frac{\ep}{4}$ in place of~$\ep$,
getting $\dt_0 > 0$.
Set $\dt = \min \left( \frac{\ep}{4}, \dt_0 \right)$.

Since $A$ is tracially $\JS$-absorbing,
there is a completely positive contractive order zero map
$\ph_0 \colon M_{n} \to A$ such that:
\begin{enumerate}
\item\label{D_6208_TracialZAbs_Small}
$1 - \ph_0 (1) \cuntz_{A} x_1$.
\item\label{D_6208_TracialZAbs_AppCm}
For $z \in M_{n}$ with $\norm{z} \leq 1$ and $a \in F$,
we have $\norm{\ph_0 (z) a - a \ph_0 (z) } < \dt$.
\setcounter{TmpEnumi}{\value{enumi}}
\end{enumerate}
Since $B$ is centrally large in~$A$,
there exist
$y_{j, k} \in A$ for $j, k = 1, 2, \ldots, n$
and $g \in B$ such that:
\begin{enumerate}
\setcounter{enumi}{\value{TmpEnumi}}
\item\label{Rel_6208_LargeSubalg_1}
$0 \leq g \leq 1$.
\item\label{Rel_6208_LargeSubalg_2}
$\norm{y_{j, k} - \ph_0 (e_{j, k}) } < \dt$
for $j, k = 1, 2, \ldots, n$.
\item\label{Rel_6208_LargeSubalg_3}
$(1 - g) y_{j, k}^* \in B$ for $j, k = 1, 2, \ldots, n$.
\item\label{Rel_6208_LargeSubalg_4}
$g \cuntz_{B} x_3$ and $g \cuntz_{A} 1$.
\item\label{Rel_6208_LargeSubalg_6}
$\norm{ g \ph_0(e_{j, k}) - \ph_0 (e_{j, k}) g } < \dt$ for $j, k = 1, 2,
\ldots, n$.
\item\label{Rel_6208_LargeSubalg_6forF}
$\norm{ g a - a g } < \dt$ for $a \in F$.
\setcounter{TmpEnumi}{\value{enumi}}
\end{enumerate}
By~(\ref{Rel_6208_LargeSubalg_3}),
for $j, k = 1, 2, \ldots, n$
we have $y_{j, k} (1 - g) \in B$,
so $\dist \big( \ph_0 (e_{j, k}) (1 - g), \, B \big) < \dt$.
Since $\dt \leq \dt_0$,
we can apply the choice of $\dt_0$
using Lemma~\ref{L_5714_MainPerturb},
taking $A$, $B$, and $\ph_0$ there to be as given
and with $x = 1 - g$.
We get a completely positive contractive order zero map
$\ph \colon M_{n} \to B$ such that:
\begin{enumerate}
\setcounter{enumi}{\value{TmpEnumi}}
\item\label{Rel_6208_MainPerturb_Dist}
$\norm{\ph (z) - \ph_0 (z) (1 - g)} < \frac{\ep}{4}$
for all $z \in M_{n}$ with
$\norm{z} \leq 1$.
\item\label{Rel_6208_MainPerturb_Cuntz}
$1 - \ph (1) \cuntz_A g \oplus [1 - \ph_0 (1)]$.
\end{enumerate}

{}From (\ref{Rel_6208_MainPerturb_Cuntz}),
(\ref{D_6208_TracialZAbs_Small}),
and~(\ref{Rel_6208_LargeSubalg_4}),
we get $1 - \ph (1) \cuntz_A x_3 \oplus x_1 \sim_A x_1 + x_3$.
Corollary~3.8 of~\cite{ArPh}
implies that $B$ is stably centrally large in~$A$.
In particular, $B$ is stably large in~$A$.
Using~(\ref{Eq_6208_Getx1x2x3}),
apply Lemma~6.5 of~\cite{PhLarge}
with $1 - \ph (1)$ in place of~$a$,
with $x$ in place of~$b$,
with $x_1 + x_3$ in place of~$c$,
and with $x_2$ in place of~$x$,
to get $1 - \ph (1) \cuntz_{B} x$.
This is part~(\ref{D_TracialZAbs_Small})
of Definition~\ref{D_TracialZAbs}.

For part~(\ref{D_TracialZAbs_AppCm})
of Definition~\ref{D_TracialZAbs},
let $a \in F$ and let $z \in M_n$ satisfy $\| z \| \leq 1$.
Then,
using (\ref{Rel_6208_MainPerturb_Dist}) at the second step,
(\ref{Rel_6208_LargeSubalg_6forF})
and (\ref{D_6208_TracialZAbs_AppCm}) at the third step,
and $\dt \leq \frac{\ep}{4}$ at the fourth step,
\begin{align*}
\norm{ [\ph (z), a] }
& \leq 2 \norm{\ph (z) - \ph_0 (z) (1 - g)}
  + \norm{ [\ph_0 (z) (1 - g), \, a] }
\\
& < \frac{\ep}{2}
    + \norm{ \ph_0 (z) } \norm{ [ 1 - g, \, a] }
    + \norm{ [\ph_0 (z), \, a] } \norm{ 1 - g }
\\
& < \frac{\ep}{2} + \dt + \dt
  \leq \ep.
\end{align*}
This completes the proof.
\end{proof}

\begin{cor}\label{C_6208_ZabsGoesDown}
Let $A$ be a simple separable infinite dimensional \uca,
and let $B$ be a centrally large nuclear subalgebra of~$A$.
If $\JS \otimes A \cong A$ then $\JS \otimes B \cong B$.
\end{cor}

\begin{proof}
The proof is essentially the same as that of
Theorem~\ref{T_5715_ZabsGoesUp},
using Theorem~\ref{T_6208_ZDown}
in place of Theorem~\ref{tZabsGoesUp}.
\end{proof}

Wei Sun~(\cite{Su}) has studied a class of
of ``generalized higher dimensional noncommutative tori''.
They are defined in terms of generators and relations,
involving $n$ commuting unitaries
(which generate a copy of $C ((S^1)^n)$),
a sequence $\te = (\te_{1}, \te_{2}, \ldots, \te_{n}) \in \R^{n}$,
a further (nonunitary) element~$b$
which satisfies commutation relations with the given unitaries
involving~$\te$,
and a nonnegative function $\gm \in C ((S^1)^n)$
related to the extent to which $b$ fails to be unitary.
The resulting algebra is called $A_{\te, \gm}$.
Let $h \colon (S^{1})^{n} \to (S^{1})^{n}$
be the \hme{}
given by
\[
h (z_{1}, z_{2}, \ldots, z_{n}) =
\big( e^{2 \pi i \te_{1}} z_{1}, \, e^{2 \pi i \te_{2}} z_{2}, \,
   \ldots, \, e^{2 \pi i \te_{n}} z_{n} \big)
\]
for $z_{1}, z_{2}, \ldots, z_{n} \in S^1$.
When $\te_{1}, \te_{2}, \ldots, \te_{n}$
are rationally independent,
so that $h$ is minimal,
and when the zero set $Y$ of $\gm$
satisfies $h^n (Y) \cap Y = \varnothing$
for all $n \in \Z \setminus \{ 0 \}$,
Wei Sun has shown that
$A_{\te, \gm}$ is isomorphic to the subalgebra
$C^{*} \big( \Z, (S^{1})^{n}, h \big)_{Y}
  \subset C^{*} \big( \Z, (S^{1})^{n}, h \big)$
as in Definition~7.3
of~\cite{PhLarge}.
Accordingly,
it is useful to obtain information
about $C^{*} \big( \Z, (S^{1})^{n}, h \big)_{Y}$
from knowledge of the structure of
$C^{*} \big( \Z, (S^{1})^{n}, h \big)$.
The algebra $C^{*} \big( \Z, (S^{1})^{n}, h \big)$
is a special example of a higher dimensional noncommutative torus.

\begin{exa}\label{E_6804_WeiSunEx}
Let the notation be as in the preceding discussion,
including the assumptions that
$\te_{1}, \te_{2}, \ldots, \te_{n}$
are rationally independent
and that $Y \subset (S^{1})^{n}$
is a closed subset satisfying $h^n (Y) \cap Y = \varnothing$
for all $n \in \Z \setminus \{ 0 \}$.
Abbreviate $C^{*} \big( \Z, (S^{1})^{n}, h \big)$ to~$A$
and $C^{*} \big( \Z, (S^{1})^{n}, h \big)_{Y}$ to~$A_{Y}$.
It is known
that $A$ is $\JS$-stable
(for example, see Corollary~3.4 of~\cite{TW}),
that $A$ has a unique tracial state~$\ta$
(Lemma 3.2(i) of~\cite{Sl}),
and that $\ta_* (K_0 (A))$ is dense in~$\R$
(this is true of its subalgebra
$A_{\te_1}$,
for which the range of the trace on K-theory
is $\Z + \te_1 \Z$ by Proposition~1.4 of~\cite{Rf0}).
Theorem~7.10 of~\cite{PhLarge} and
Theorem~4.6 of~\cite{ArPh}
imply that $A_Y$ is a centrally large subalgebra of~$A$.
So Corollary~\ref{C_6208_ZabsGoesDown} implies that $A_Y$
is $\JS$-stable.
It now follows from Theorem~6.7 of~\cite{Rrd}
that $A_Y$ has stable rank one.
Theorem~6.2 of~\cite{PhLarge}
implies that $A_Y$ has a unique tracial state,
namely $\sm = \ta |_{A_Y}$.
Suppose further that $K^1 (Y) = 0$.
It then follows from Theorem~2.4 of~\cite{Pt4}
and the discussion
after Example~2.6 of~\cite{Pt4}
that the map from $K_0 (A_Y)$ to $K_0 (A)$
is surjective.
So $\sm_* (K_0 (A)_Y)$ is dense in~$\R$.
Now Corollary~7.3 of~\cite{Rrd} implies that,
in this case, $A_Y$ has real rank zero.
In particular,
under these hypotheses,
the algebra $A_{\te, \gm}$ of~\cite{Su}
is $\JS$-stable,
has stable rank one,
and, if $K^1 (Y) = 0$, has real rank zero.
\end{exa}

The hypothesis $K^1 (Y) = 0$
is stronger than needed for the conclusion
that $A_Y$ has real rank zero.
Indeed, under our other hypotheses,
this may always be true.
We leave such questions to~\cite{Su}.

There are other ways to get these results.
For example,
$C^{*} \big( \Z, (S^{1})^{n}, h \big)_{Y}$
is a simple direct limit,
with no dimension growth,
of recursive subhomogeneous \ca{s}.

In Theorem~\ref{T_5715_ZabsGoesUp},
it is not possible to replace $\JS$
with a general strongly selfabsorbing \ca~$D$.

\begin{exa}\label{E_5715_UAbsNotUp}
Let $D$ be the $2^{\infty}$~UHF algebra.
Let $X$ be the Cantor set,
and let $h \colon X \to X$
be the $2$-odometer.
(See page~332 of~\cite{Pt1},
or Section VIII.4 of~\cite{Dv}.)
Set $A = C^* (\Z, X, h)$,
fix $y \in X$,
and take $B = C^* (\Z, X, h)_{\set{y}}$,
as in Definition 7.3 of~\cite{PhLarge}.
(This is the algebra called $A_{ \set{h (y)}}$
in Theorem~3.3 of~\cite{Pt1}.
One uses $h (y)$ rather than $y$ because
of the difference between the conventions used in \cite{Pt1}
and~\cite{PhLarge}.)
Then $K_0 (A) \cong \Z \big[ \frac{1}{2} \big]$,
as explained on page~332 of~\cite{Pt1}.
The inclusion of $B$ in $A$ is an isomorphism
on~$K_0$,
by Theorem~4.1 of~\cite{Pt1},
and $B$ is an AF~algebra
by Theorem~3.3 of~\cite{Pt1},
so $B$ is the $2^{\infty}$~UHF algebra.
(This specific case is Example VIII.6.3 in~\cite{Dv}.)
In particular,
$B$ is $D$-absorbing.
However,
the Pimsner-Voiculescu exact sequence
(see Theorem VIII.5.1 of~\cite{Dv})
easily shows that $K_1 (A) \cong \Z$.
The K\"{u}nneth Theorem
(Theorem~4.1 of~\cite{Sc2})
therefore implies that $A$ is not $D$-absorbing.
The subalgebra $B$ is centrally large in~$A$
by Theorem~7.10 of~\cite{PhLarge}
and Theorem~4.6 of~\cite{ArPh}.
\end{exa}

\renewcommand{\bibname}{\textsc{REFERENCES}}

\end{document}